\documentclass[a4paper,12pt]{article}
\usepackage{amsmath,amsthm,amsfonts,amssymb,enumitem,graphicx,float,calc,microtype,mathtools,thmtools,underscore,anyfontsize,todonotes,fontenc,lmodern}
\usepackage{caption,subcaption}
\usepackage{latexsym}
\usepackage{tikz}
\usepackage{xcolor}
\usepackage[english]{babel}
\usepackage[lmargin=30mm,rmargin=30mm,tmargin=30mm,bmargin=30mm]{geometry}
\renewcommand{\baselinestretch}{1.14}
\allowdisplaybreaks

\usepackage[utf8]{inputenc}
\usepackage[T1]{fontenc}
\usepackage{xcolor}
\usepackage{paralist}
\usepackage{enumitem}
\usepackage[affil-it]{authblk}
\usepackage[numbers,sort&compress,longnamesfirst]{natbib}
\def\NAT@spacechar{~}
\makeatother

\usepackage[pdftitle={??},
pdfauthor={Hickingbotham -- Wood},
colorlinks =true,
linkcolor={black},
urlcolor={blue!80!black},
filecolor={blue!80!black},
citecolor={black},
anchorcolor=blue,
filecolor=blue, 
menucolor=blue]{hyperref}

\usepackage[noabbrev,capitalise]{cleveref}
\crefname{lem}{Lemma}{Lemmas}
\crefname{thm}{Theorem}{Theorems}
\crefname{cor}{Corollary}{Corollaries}
\crefname{prop}{Proposition}{Propositions}
\crefname{conj}{Conjecture}{Conjectures}
\crefname{open}{Open Problem}{Open Problems}
\crefname{obs}{Observation}{Observations}

\crefformat{equation}{(#2#1#3)}
\Crefformat{equation}{Equation #2(#1)#3}

\theoremstyle{plain}
\newtheorem{thm}{Theorem}
\newtheorem{lem}[thm]{Lemma}
\newtheorem{cor}[thm]{Corollary}

\theoremstyle{definition}

\DeclarePairedDelimiter{\ceil}{\lceil}{\rceil}

\DeclareMathOperator{\tw}{tw}
\DeclareMathOperator*{\pw}{pw}
\DeclareMathOperator{\td}{td}

\renewcommand{\leq}{\leqslant}

\renewcommand{\geq}{\geqslant}

\theoremstyle{definition}

\DeclareMathOperator*{\bs}{\backslash}
\DeclareMathOperator*{\sse}{\subseteq}

\DeclareMathOperator*{\Pcal}{\mathcal{P}}
\DeclareMathOperator*{\Qcal}{\mathcal{Q}}

\presetkeys%
{todonotes}%
{inline}{}

\title{\bf Structural Properties of Bipartite Subgraphs}
\author{%
	Robert Hickingbotham,\!\!%
	\thanks{School of Mathematics, Monash University, Melbourne, Australia (\texttt{robert.hickingbotham@monash.edu}). Research supported by an Australian Government Research Training Program Scholarship.}
	\,\,
	David R. Wood\thanks{School of Mathematics, Monash University, Melbourne, Australia (\texttt{david.wood@monash.edu}). Research supported by the Australian Research Council.}
}

\begin{document}
\maketitle
\date{}
\begin{abstract}
	This paper establishes sufficient conditions that force a graph to contain a bipartite subgraph with a given structural property. In particular, let $\beta$ be any of the following graph parameters: Hadwiger number, Haj\'{o}s number, treewidth, pathwidth, and treedepth. In each case, we show that there exists a function $f$ such that every graph $G$ with $\beta(G)\geq f(k)$ contains a bipartite subgraph $\hat{G}\sse G$ with $\beta(\hat{G})\geq k$. 
\end{abstract}
\section{Introduction}
A classical result by Erd\H{o}s \cite{erdos1965some} states that every graph contains a bipartite subgraph with at least half of the edges and that this is best possible for complete graphs. Many papers study the maximum number of edges in a bipartite subgraph if the graph satisfies some structural property such as being triangle-free (see \cite{alon1996bipartite, ABKS2003cut, AKS2005max, edwards1973extremal, erdos1988make, sudakov2007making}). Instead of maximising the number of edges, we consider other structural properties of bipartite subgraphs. In general, we are interested in sufficient conditions that force a graph to have a bipartite subgraph with a given property. Typically, we study this question for a graph parameter $\beta$.\footnote{A \emph{graph parameter} is a function $\beta$ such that $\beta(G)\in\mathbb{R}$ for every graph $G$ and $\beta(G_1)=\beta(G_2)$ for all isomorphic graphs $G_1$ and $G_2$.} In such cases, the first question is whether there exists a function $f$ such that every graph $G$ with $\beta(G)\geq f(k)$ contains a bipartite subgraph $\hat{G}\sse G$ with $\beta(\hat{G})\geq k$. We show that such a function exists for each of the following parameters: Hadwiger number (\Cref{Minors}), Haj\'{o}s number (\Cref{Minors}), treewidth (\Cref{Treewidth}), pathwidth (\Cref{PathwidthTreedepth}), and treedepth (\Cref{PathwidthTreedepth}). These parameters are important in structural graph theory, especially in Robertson-Seymour graph minor theory. The second question for such parameters is what is the least possible function $f$. We discuss this question for each of these parameters in the relevant section.

One motivation for exploring the properties of bipartite subgraphs is their application in the study of direct products.\footnote{For graphs $G_1$ and $G_2$, the \emph{direct product}, $G_1 \times G_2$, is defined with vertex set $V(G_1) \times V(G_2)$ where $(u_1,u_2)(v_1,v_2) \in E(G_1 \times G_2)$ if $u_1v_1 \in E(G_1)$ and $u_2v_2 \in E(G_2)$.} For example, a graph $G$ is not necessarily a minor of $G \times K_2$ \cite{bottreau1998some}, but it is easy to see that every bipartite subgraph of $G$ is a subgraph of $G \times K_2$. So in studying the structural properties of $G \times K_2$, it is very helpful to know about the behaviour of bipartite subgraphs in $G$. We explore this line of research in our companion papers \cite{HW2021products,HW2021hadwiger}.

\subsection{Relevant Literature}
For integers $a,b\geq 1$ where $a\leq b$, define $[b]:=\{1,\dots,b\}$ and $[a,b]:=\{a,a+1,\dots,b-1,b\}$. We now mention several results from the literature concerning structural properties of bipartite subgraphs. For minimum degree, it is folklore that every graph with minimum degree at least $2d$ contains a bipartite subgraph with minimum degree at least $d$ (see \cite{scott2021separation} for a proof). 

For connectivity, \citet{mader1972existenzn} showed that every graph with sufficiently large average degree contains a highly connected subgraph. In particular, for every integer $k \geq 1$, every graph with average degree at least $4k$ contains a $(k+1)$-connected subgraph. Now if a graph is $8k$-connected, then its average degree is at least $8k$. By Erd\H{o}s' result \cite{erdos1965some}, the graph contains a bipartite subgraph with average degree at least $4k$. Hence, Mader's and Erd\H{o}s' results together imply the following.

\begin{thm}
	For every integer $k \geq 1$, every graph with average degree at least $8k$ (in particular, every $8k$-connected graph) contains a $(k+1)$-connected bipartite subgraph.
\end{thm}
 See \cite{BK2016subgraphs,yuster2003subgraphs} for improved bounds to Mader's result.
\section{Constructing Bipartite Subgraphs}
Suppose we have a proper black-white colouring of a graph. To \emph{switch} the colouring means to recolour the black vertices white and the white vertices black. For a graph $G$ with vertex-disjoint subgraphs $H_1$ and $H_2$, a path $P=(v_1,\dots,v_n)$ in $G$ \emph{joins} $H_1$ and $H_2$ if $V(P) \cap V(H_1)=\{v_1\}$ and $V(P) \cap V(H_2)=\{v_n\}$. The next lemma allows us to construct bigger bipartite subgraphs from smaller ones. 
 
\begin{lem}\label{MainLemmaConstructBipartiteSubgraphs}
	Let $G$ be a graph and let $H_1,\dots,H_t$ be vertex-disjoint bipartite subgraphs in $G$. Suppose there exists a set of internally disjoint paths $\Pcal=\{P_1,\dots, P_k\}$ in $G$ such that for all $i \in [k]$, $P_i$ joins $H_j$ and $H_{\ell}$ for some distinct $j,\ell \in [t]$ and that $P_i$ and $H_b$ are disjoint for all $b \in [t]\bs \{j,\ell\}$. Then there exists a subset $\Qcal \sse \Pcal$ where $|\Qcal|\geq k/2$ such that the graph $\hat{G}:=(\bigcup_{i \in [t]} H_i) \cup (\bigcup_{Q \in \Qcal}Q)$ is a bipartite subgraph of $G$.
\end{lem}

\begin{proof}
	For all $j \in [t]$, let $\Pcal^{(j)}$ be the set of paths in $\Pcal$ that join any $H_i$ and $H_{i'}$ where $i,i' \in [j]$. Note that $\emptyset=\Pcal^{(1)}\sse \Pcal^{(2)} \sse \dots \sse \Pcal^{(t)}=\Pcal$. We prove the following by induction on $j$.
	
	\textbf{Claim}: For all $j \in [t]$, there exists a subset $\Qcal^{(j)} \sse \Pcal^{(j)}$ where $|\Qcal^{(j)}|\geq |\Pcal^{(j)}|/2$ such that the graph $\hat{G}^{(j)}:=(\bigcup_{i \in [j]} H_i) \cup (\bigcup_{Q \in \Qcal^{(j)}}Q)$ is a bipartite subgraph of $G$.
	
	For $j=1$, the claim holds trivially since $\Pcal^{(1)}=\emptyset$. Now assume it holds for $j-1$. Then there exists a set of paths $\Qcal^{(j-1)}$ with size at least $|\Pcal^{(j-1)}|/2$ such that $\hat{G}^{(j-1)}$ is bipartite. Since $\hat{G}^{(j-1)}$ and $H_j$ are vertex-disjoint bipartite subgraphs, we may take a proper black-white colouring of their vertices. Let $A^{(j)}:= \Pcal^{(j)}\bs\Pcal^{(j-1)}$. Then each path in $A^{(j)}$ joins $H_j$ to some $H_i$ where $i \in [j-1]$. We say that a path $P=(v_1,\dots,v_n) \in A^{(j)}$ is \emph{agreeable} if there exists a proper black-white colouring of the vertices of $P$ such that the colouring of $v_1$ and $v_n$ corresponds with the colour they were previously assigned. Otherwise we say that $P$ is \emph{disagreeable}. Observe that if we switch the black-white colouring for the vertices in $H_j$, then all the agreeable paths in $A^{(j)}$ become disagreeable and all the disagreeable paths become agreeable. So if there are more disagreeable paths than agreeable, then switch the black-white colouring of the vertices in $H_j$. In doing so, the set of agreeable paths $B^{(j)}\sse A^{(j)}$ has size at least $|A^{(j)}|/2$. Let $\Qcal^{(j)}:=\Qcal^{(j-1)} \cup (\bigcup_{P \in B^{(j)}}P)$. Then $\hat{G}^{(j)}$ is bipartite and $|\Qcal^{(j)}|=|\Qcal^{(j-1)} \cup B^{(j)}| \geq | \Pcal^{(j-1)}|/2+|A^{(j)}|/2=|\Pcal ^{(j)}|/2$ as required. 
	
	The result follows when $j=t$.
\end{proof}
Note that \Cref{MainLemmaConstructBipartiteSubgraphs} generalises Erd\H{o}s' result \cite{erdos1965some}: if $H_1,\dots,H_t$ are the vertices of a graph $G$ and $\Pcal$ is the set of the edges in $G$, then by \Cref{MainLemmaConstructBipartiteSubgraphs}, $G$ contains a bipartite subgraph with at least half the edges.

\section{Minors}\label{Minors}
A graph $H$ is a \emph{minor} of a graph $G$ if a graph isomorphic to $H$ can be obtained from $G$ by vertex deletion, edge deletion, and edge contraction. The study of graph minors within the broader field of structural graph theory is a fundamental topic of research that was initiated by the seminal work of Robertson and Seymour. We now prove a sufficient condition under which a graph has a bipartite subgraph that contains a given graph as a minor. For a graph $G$, let $\nabla(G)$ be the maximum average degree of a minor of $G$. For a graph $H$, let $f(H)$ be the infimum of all real numbers $d$ such that every graph with average degree at least $d$ contains $H$ as a minor. \citet{mader1968homorphies} showed that $f(H)$ exists for every graph $H$. The function $f(H)$ is called the \emph{extremal function for $H$ minors} and has been studied for many graphs 
(see \cite{chudnovsky2011edge, jorgensen1994K8, KP2010K3t, KT2012Kst, KT2008Kst, ST2006K9, CLNWY2017disconnected, dirac19084homomorphism, HW2015cycles, HW2016average, HW2018petersen, kostochka1982mean, kostochka1984average, mader1968homorphies, MT2005noncomplete, RW2016sparse, TY2019triangle, thomason1984contraction, thomason2001extremal, wager1937uber}).

\begin{thm}\label{MainNablaExtremal}
	For all graphs $G$ and $H$, if $\nabla(G) \geq 2 f(H)$ then $G$ contains a bipartite subgraph $\hat{G}\sse G$ where $H$ is a minor of $\hat{G}$.
\end{thm}

\begin{proof}
	Let $H_1$ be a minor of $G$ with average degree at least $2f(H)$ and let $G_1$ be a minimal subgraph of $G$ that contains $H_1$ as a minor. By minimality, $G_1$ consists of pairwise disjoint subtrees $T_1,\dots,T_t$ of $G$ such that for each $ij \in E(H_1)$, there exists an edge $v_iv_j \in E(G_1)$ where $v_i \in V(T_i)$ and $v_j \in V(T_j)$. Let $\Pcal:=\{v_iv_j:ij \in E(H_1)\}$. Since $T_1,\dots,T_t$ are bipartite, by \Cref{MainLemmaConstructBipartiteSubgraphs} there exists a subset $\Qcal\sse \Pcal$ where $|\Qcal|\geq |\Pcal|/2$ such that the subgraph $\hat{G}:=(\bigcup_{i \in [t]} T_i)\cup (\bigcup_{Q \in \Qcal}Q)$ is bipartite. 
	
	Let $H_2$ be the minor of $\hat{G}$ obtained by contracting each $T_i$ into a vertex. Then $|V(H_2)|=|V(H_1)|$ and $|E(H_2)|=|\Qcal|\geq |E(H_1)|/2$. As such, the average degree of $H_2$ is at least $f(H)$. By the definition of the extremal function, $H$ is a minor of $H_2$ and hence $H$ is a minor of $\hat{G}$. 
\end{proof}

A graph $H$ is a \emph{topological-minor} of a graph $G$ if a subgraph of $G$ is isomorphic to a subdivision of $H$. For a graph $G$, let $\tilde {\nabla}(G)$ be the greatest average degree of a topological-minor of $G$. For a graph $H$, let $\tilde{f}(H)$ be the infimum of all real numbers $d$ such that every graph with average degree at least $d$ contains $H$ as a topological-minor. \citet{mader1967topological} proved that $\tilde{f}(H)$ exists for every graph $H$. The function $\tilde {f}(H)$ is called the \emph{extremal function for topological $H$ minors} (see \cite{BT1998topological,KJS1994topological,KJS1996topological,KO2002topological, KO2004cycle,KO2006cliques,kuhn2006extremal, mader1968homorphies,EH1964topological,jung1970eine}). Using a similar approach to our proof of \Cref{MainNablaExtremal}, we prove an analogous result for topological-minors.
 
\begin{thm}\label{MainNablaExtremalTopological}
	For all graphs $G$ and $H$, if $\tilde {\nabla}(G) \geq 2 \tilde {f}(H)$ then $G$ contains a bipartite subgraph $\hat{G}\sse G$ where $H$ is a topological-minor of $\hat{G}$.
\end{thm}

\begin{proof}
	Let $H_1$ be an $n$-vertex graph that is a topological-minor of $G$ with average degree at least $2\tilde{f}(H)$. Let $G_1$ be a subgraph of $G$ that is isomorphic to a subdivision of $H_1$. Then $G_1$ consists of a set of $n$ vertices $X=\{x_1,\dots,x_n\}$ such that for each $ij \in E(H_1)$, there exists an $(x_i,x_j)$-path $P_{i,j}$ in $G_1$ where $V(P_{i,j})\cap X=\{x_i,x_j\}$ and these paths are internally disjoint. Let $\Pcal:=\{P_{i,j}:ij \in E(H_1)\}$. By \Cref{MainLemmaConstructBipartiteSubgraphs}, there exists $\Qcal\sse \Pcal$ where $|\Qcal|\geq |\Pcal|/2$ such that the subgraph $\hat{G}:=X\cup (\bigcup_{Q \in \Qcal}Q)$ is bipartite. 
	
	Let $H_2$ be the topological-minor of $\hat{G}$ that is obtained by contracting each $P_{i,j}\in \Qcal$ into an edge. Then $|V(H_2)|=|V(H_1)|$ and $|E(H_2)|=|\Qcal|\geq |E(H_1)|/2$. As such, the average degree of $H_2$ is at least $\tilde{f}(H)$. By the definition of $\tilde{f}(H)$, $H_2$ contains $H$ as a topological-minor and hence $H$ is a topological-minor of $\hat{G}$.
\end{proof}

We now derive results concerning complete-minors and complete-topological-minors in bipartite subgraphs. For a graph $G$, the \emph{Hadwiger number}, $\eta(G)$, of $G$ is the maximum integer $t$ such that $K_t$ is a minor of $G$. Similarly, the \emph{Haj\'{o}s number}, $\tilde{\eta}(G)$, of $G$ is the maximum integer $t$ such that $K_t$ is a topological-minor of $G$. \citet{kostochka1982mean, kostochka1984average} and \citet{thomason1984contraction} independently proved that $f(K_t)\in \Theta(t\sqrt{\log(t)})$.
\begin{thm}[\cite{kostochka1982mean,kostochka1984average,thomason1984contraction}]\label{AvgDegForceMinor}
	For every integer $t\geq 1$, we have $f(K_t)\leq (c+o(1))t\sqrt{\log(t)}$ for some positive constant $c$.
\end{thm}
Note that \citet{thomason2001extremal} subsequently determined the best possible $c$. \Cref{MainNablaExtremal,AvgDegForceMinor} immediately imply the following.

\begin{thm}\label{MainHadwigerBipartiteSubgraph}
	For every integer $t\geq 1$ and graph $G$, if $\eta(G)\geq (2c+o(1))t\sqrt{\log(t)}$ then $G$ contains a bipartite subgraph $\hat{G}\sse G$ such that $\eta(\hat{G})\geq t$ where $c$ is the constant from \Cref{AvgDegForceMinor}.
\end{thm}

 For the extremal function for complete-topological-minors, \citet{BT1998topological} and \citet{KJS1996topological} independently proved that $\tilde{f}(K_t)= \Theta(t^2)$.
\begin{thm}[\cite{BT1998topological,KJS1996topological}]\label{AvgDegForceTopMinor}
	For every integer $t\geq 1$, we have $\tilde{f}(K_t)\leq (c+o(1))t^2$ for some positive constant $c$.
\end{thm}

See \cite{kuhn2006extremal} for the best known bounds on $c$. \Cref{MainNablaExtremalTopological,AvgDegForceTopMinor} imply the following.

\begin{thm}\label{MainHajosNumber}
	For every integer $t\geq 1$ and graph $G$, if $\tilde{\eta}(G)\geq (2c+o(1))t^2$, then $G$ contains a bipartite subgraph $\hat{G} \sse G$ with $\tilde{\eta}(\hat{G}) \geq t$ where $c$ is the constant from \Cref{AvgDegForceTopMinor}.
\end{thm}

It is open whether the bounds from \Cref{MainHadwigerBipartiteSubgraph,MainHajosNumber} are asymptotically best possible.

We now consider \emph{shallow minors}. For an integer $r\geq 0$ and graph $G$, let $\nabla_r(G)$ be the greatest average degree of a minor of $G$ that can be obtained by contracting pairwise disjoint subgraphs of radius at most $r$ in $G$ into vertices and then taking a subgraph. The function $\nabla_r$ is of particular importance in graph sparsity theory (see \cite{nevsetvril2012sparsity}).

\begin{thm}\label{MainExpansionFunction}
	For every integer $r\geq 0$, every graph $G$ contains a bipartite subgraph $\hat{G}_r \sse G$ such that $\nabla_r(\hat{G}_r)\geq \frac{1}{2}\nabla_r(G)$.
\end{thm}

\begin{proof}
	Let $H$ be a graph with average degree $\nabla_r(G)$ that can be obtained from $G$ by contracting pairwise disjoint subgraphs of radius at most $r$ in $G$ into vertices and then taking a subgraph. Let $G_1$ be a minimal subgraph of $G$ from which $H$ can be obtained by contracting subgraphs in $G_1$ of radius at most $r$ into vertices. By minimality, $G_1$ consists of pairwise disjoint trees $T_1,\dots,T_t$ each with radius at most $r$ such that for each $ij \in E(H)$, there exists an edge $v_iv_j \in E(G_1)$ where $v_i \in V(T_i)$ and $v_j \in V(T_j)$. Let $\Pcal:=\{v_iv_j:ij \in E(H)\}$. Since $T_1,\dots,T_t$ are bipartite, by \Cref{MainLemmaConstructBipartiteSubgraphs} there exists $\Qcal\sse \Pcal$ where $|\Qcal|\geq |\Pcal|/2$ such that the subgraph $\hat{G}_r:=(\bigcup_{i \in [t]} T_i)\cup (\bigcup_{Q \in \Qcal}Q)$ is bipartite. 
	
	For each $i \in V(H)$, contract the $T_i$ subgraph in $\hat{G}_r$ into a vertex and let $\hat{H}$ be the graph obtained. Now $|V(\hat{H})|=|V(H)|$ and $|E(\hat{H})|=|\Qcal|\geq |E(H)|/2$. As such, the average degree of $\hat{H}$ is at least $\frac{1}{2}\nabla_r(G)$ and since $\hat{H}$ is obtained from $\hat{G}_r$ by contracting pairwise disjoint subgraphs with radius at most $r$, it follows that $\nabla_r(\hat{G}_r)\geq \frac{1}{2}\nabla_r(G)$.
\end{proof}

It is open whether there exists a function $g$ such that every graph $G$ contains a bipartite subgraph $\hat{G}\sse G$ where for every integer $r \geq 0$, if $\nabla_r(G) \geq g(k)$ then $\nabla_r(\hat{G})\geq k$.

\section{Treewidth}\label{Treewidth}
For a graph $G$, a \emph{tree-decomposition} of $G$, $(T,\{B_t\sse V(G)\}_{t \in V(T)})$, is a pair that satisfies the following properties: 

\begin{compactitem}
	\item $T$ is a tree;
	\item For all $v \in V(G)$, $T[\{t \in V(T): v \in B_t\}]$ is a non-empty subtree of $T$; and
	\item For all $uv \in E(G)$, there is a node $t \in V(T)$ such that $u,v \in B_t$.
\end{compactitem}

The \emph{width} of a tree-decomposition is $\max\{|B_t|-1:t \in V(T)\}$. The \emph{treewidth} of $G$, $\tw(G)$, is the minimum width of a tree-decomposition of $G$. A \emph{path-decomposition} of $G$ is a tree-decomposition of $G$ where the tree is a path. The \emph{pathwidth} of $G$, $\pw(G)$, is the minimum width of a path-decomposition of $G$. Treewidth and pathwidth respectively measure how similar a graph is to a tree and to a path, and are important parameters in algorithmic and structural graph theory (see \cite{bodlaender1993tourist,HW2017tied,reed1997treewidth}). Tree-decompositions and path-decompositions were introduced by Robertson and Seymour \cite{robertson1986algorithmic,robertson1983graph}. 

The following theorem is our key contribution in this section.

\begin{thm}\label{MaintwBipartiteSubgraph}
	There exists a function $f$ such that every graph $G$ with $\tw(G)\geq f(k)$ contains a bipartite subgraph $\hat{G}\sse G$ such that $\tw(\hat{G})\geq k$.
\end{thm}

Before proving \Cref{MaintwBipartiteSubgraph}, we need to introduce two families of graphs that are relevant to the proof. The first are grids. For integers $j,k\geq 1$, the $(j \times k)$-grid, $G_{j \times k}$, is the graph with vertex set $V(G_{j \times k})=\{(v_1,v_2):v_1\in[j],v_2 \in [k]\}$ where $(v_1,v_2)(u_1,u_2)\in E(G_{j \times k})$ if $v_1=u_1$ and $|v_2-u_2|=1$, or $v_2=u_2$ and $|v_1-u_1|=1.$ The edge is \emph{vertical} when $v_1=u_1$, and is \emph{horizontal} when $v_2=u_2$. For $i \in [k]$, the \emph{$i^{th}$-horizontal path} in the grid is the unique path from $(1,i)$ to $(j,i)$ where each edge in the path is horizontal. 

The second family are walls. For integers $j,k\geq 1$ and $i \in [k]$, let $\bar{P}^{(i)}=(v_1^{(i)},\dots, v_{2j}^{(i)})$ be a path on $2j$ vertices. The $(2j \times k)$-wall, $W_{2j \times k}$, consists of the paths $\bar{P}^{(1)},\dots,\bar{P}^{(k)}$ where for each $\ell \in [j]$ and $i \in [2,k]$, we add the edge $v_{2\ell}^{(i-1)}v_{2\ell}^{(i)}$ whenever $i$ is even and the edge $v_{2\ell-1}^{(i-1)}v_{2\ell-1}^{(i)}$ whenever $i$ is odd (see \Cref{fig:wall}). For every $i \in [k]$ and $\ell \in [2j]$, we call $\bar{P}^{(i)}$ a \emph{horizontal path} in the wall and $v_{\ell}^{(i-1)}v_{\ell}^{(i)}$ a \emph{vertical edge} (if it exists). Observe that the $(2j \times k)$-wall is a spanning subgraph of the $(2j \times k)$-grid where for each $i \in [k]$, $\bar{P}^{(i)}$ is the $i^{th}$-horizontal path in $G_{2j \times k}$ and every vertical edge in $W_{2j \times k}$ corresponds to a vertical edge in $G_{2j \times k}$.

\begin{figure}[!h]
	\centering\includegraphics[width=0.45\textwidth]{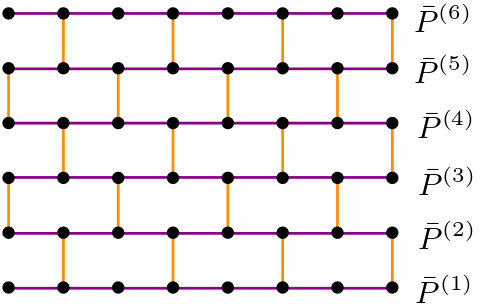}
	\caption{The $(8 \times 6)$-wall, $W_{8 \times 6}$.}
	\label{fig:wall}
\end{figure}
\citet{robertson1986planar} showed that grids are a canonical example of graphs with large treewidth.

\begin{thm}[\cite{robertson1986planar}]\label{ExcludedGridMinorTheorem}
	There exists a function $f$ such that every graph with treewidth at least $f(k)$ contains $G_{k \times k}$ as a minor. Moreover, $\tw(G_{k \times k})= k$.
\end{thm}

In a recent breakthrough, \citet{CC2016polynomial} showed that the function $f$ in \Cref{ExcludedGridMinorTheorem} is polynomial (see \cite{CT2021excluded} for the current best upper bound as well as \cite{chuzhoy2015grid,chuzhoy2016improved}). 

Now $G_{k \times k}$ is a minor of $W_{2k \times 2k}$ and hence $\tw(W_{2k \times 2k})\geq k$. Furthermore, since walls have maximum degree $3$, $W_{2k \times 2k}$ is a minor of a graph $G$ if and only if $W_{2k \times 2k}$ is a topological-minor of $G$. Hence, \Cref{ExcludedGridMinorTheorem} implies the following.

\begin{cor}\label{TreewidthWall}
	There exists a function $f$ such that every graph with treewidth at least $f(k)$ contains $W_{2k \times 2k}$ as a topological-minor. Moreover, $\tw(W_{2k \times 2k})\geq k$.
\end{cor}

We now show that every subdivision of a sufficiently large wall contains a bipartite subgraph that is a subdivision of a large wall. This lemma is the heart of the proof for \Cref{MaintwBipartiteSubgraph}.

\begin{lem}\label{BipartiteSubgraphWall}
	Every graph $G$ that is a subdivision of the $(k2^{k+1} \times k)$-wall, $W_{k2^{k+1} \times k}$, contains a bipartite subgraph $\hat{G} \sse G$ which is a subdivision of the $(2k\times k)$-wall.
\end{lem}
\begin{proof}
	Let ${P}^{(1)},\dots,{P}^{(k)}$ be the horizontal paths in $G$ where each ${P}^{(i)}$ is a subdivision of the path $\bar{P}^{(i)}=(v_1^{(i)},\dots, v_{k2^{k+1}}^{(i)})$ in $W_{k2^{k+1} \times k}$. For even $i \in [2,k]$ and $\ell \in [k2^k]$, let $S^{(i)}_\ell$ be the path in $G$ that corresponds to the subdivided vertical edge $v_{2\ell}^{(i-1)}v_{2\ell}^{(i)}$ in $W_{k2^{k+1} \times k}$ and for odd $i \in [2,k]$, let $S^{(i)}_\ell$ be the path in $G$ that corresponds to the subdivided edge $v_{2\ell-1}^{(i-1)}v_{2\ell-1}^{(i)}$ in $W_{k2^{k+1} \times k}$. To simplify our notation, let $S^{(1)}_\ell$ be an empty subgraph of $G$ for all $\ell \in [k2^k]$. 
	
	Initialise $A_1:=[k2^k]$. We construct $A_1 \supseteq A_2 \supseteq \dots \supseteq A_k$ where $|A_{i}|\geq k2^{k-i}$ for all $i \in [k]$ such that $H^{(j)}:=\bigcup_{i \in [j]} (P^{(i)} \cup (\bigcup_{\ell \in A_j}S_\ell^{(i)}))$ is a bipartite subgraph of $G$ for all $j \in [k]$. We proceed by induction on $j$. For $j=1$, $H^{(1)}=P^{(1)}$ and hence it is a bipartite subgraph of $G$. Now suppose that $j>1$. By induction, $H^{(j-1)}$ is bipartite. Let $\mathcal{S}^{(j)}:=\{S_{\ell}^{(j)}:\ell \in A_{j-1}\}$. For distinct $S,S' \in \mathcal{S}^{(j)}$, the path $S$ joins $H^{(j-1)}$ to $P^{(j)}$, and the paths $S$ and $S'$ are internally disjoint. Since $P^{(j)}$ and $H^{(j-1)}$ are vertex-disjoint bipartite subgraphs of $G$, by \Cref{MainLemmaConstructBipartiteSubgraphs} there exists $\Qcal^{(j)}\sse \mathcal{S}^{(j)}$ where $|\Qcal^{(j)}|\geq |\mathcal{S}^{(j)}|/2$ such that the graph $\hat{H}^{(j)}:=H^{(j-1)}\cup P^{(j)} \cup (\bigcup_{S \in \Qcal^{(j)}} S)$ is bipartite (see \Cref{fig:subdividedwall}). By setting $A_j:=\{\ell \in A_{j-1}:S^{(j)}_{\ell} \in \Qcal^{(j)}\}$, it follows that $|A_j|=|\Qcal^{(j)}|\geq |A_{j-1}|/2\geq k2^{k-j}$ and the subgraph $H^{(j)}$ is bipartite.

\begin{figure}[!htb]
	\begin{center}
		\includegraphics[width=0.95\linewidth]{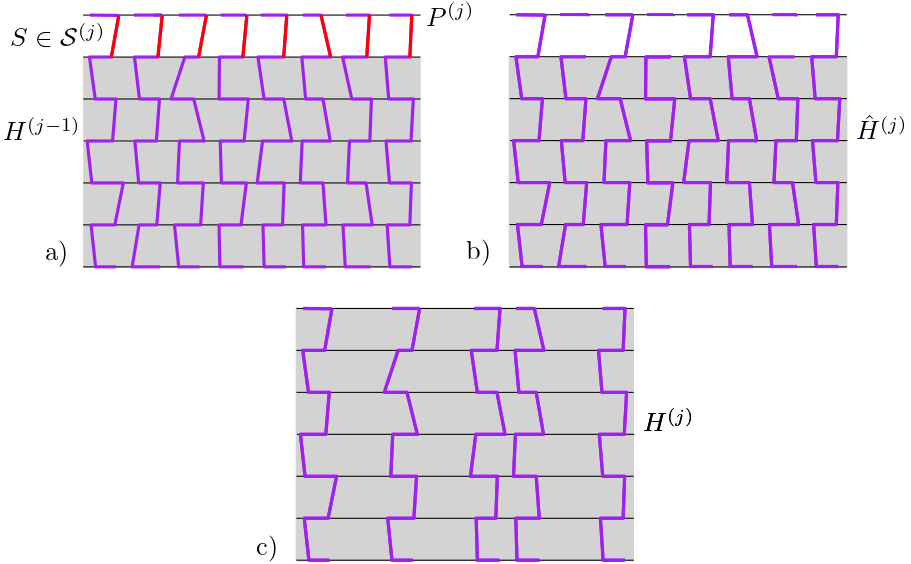}
		\caption{The inductive proof of \Cref{BipartiteSubgraphWall}: a) the set $\mathcal{S}^{(j)}$ of disjoint paths that join $H^{(j-1)}$ to $P^{(j)}$; b) the bipartite subgraph $\hat{H}^{(j)}$; c) the bipartite subgraph $H^{(j)}$.}
		\label{fig:subdividedwall}
	\end{center}
\end{figure}

	Now consider $H^{(k)}$. We have $|A_k| \geq k$. Let $\ell_1,\dots,\ell_k \in A_k$ such that $\ell_i < \ell_{i'}$ whenever $i < i'$. For each $i \in [k]$, let $\hat{P}^{(i)}$ be the subpath of $P^{(i)}$ from $v_{2\ell_1-1}^{(i)}$ to $v_{2\ell_k}^{(i)}$. Let $\hat{G}:=\bigcup_{i \in [k]} (\hat{P}^{(i)} \cup (\bigcup_{j \in [k]}S_{\ell_j}^{(i)}))$. Then $\hat{G}$ is bipartite as it is a subgraph of $H^{(k)}$. We now explain how $\hat{G}$ is a subdivision of the $(2k \times k)$-wall. For each of the following pairs of vertices in $\hat{G}$, there is a path between them which consists solely of horizontal edges: $(v_{2\ell_{j}-1}^{(i)}, v_{2\ell_j}^{(i)})$ where $i,j \in [k]$; and $(v_{2\ell_j}^{(i)},v_{2\ell_{j+1}-1}^{(j)})$ where $i \in [k], j \in [k-1]$. The union of these horizontal paths with the $S_{\ell_j}^{(i)}$ paths where $i,j \in [k]$ defines a subdivision of the $(2k\times k)$-wall, as required.
\end{proof}

\Cref{MaintwBipartiteSubgraph} directly follows from \Cref{TreewidthWall,BipartiteSubgraphWall}. Note that due to \Cref{BipartiteSubgraphWall}, the function $f$ in \Cref{MaintwBipartiteSubgraph} is exponential. It is open whether this can be improved to a polynomial function. The key challenge is whether every subdivision of a wall contains a bipartite subgraph with polynomially large treewidth.

\subsection{Pathwidth and Treedepth}\label{PathwidthTreedepth}
We now move on to consider pathwidth and treedepth. The \emph{closure} of a rooted tree $T$ is the graph obtained from $T$ by adding the edges $uv$ whenever $u$ is an ancestor of $v$. The \emph{closure} of a rooted forest is the closure of each of the rooted trees in the forest. The \emph{height} of a rooted tree is the number of vertices in the longest root to leaf path. The \emph{height} of a rooted forest is the maximum height of its component subtrees. For a graph $G$, the \emph{treedepth}, $\td(G)$, of $G$ is the minimum height of a rooted forest whose closure contains $G$ as a subgraph. Like treewidth and pathwidth, treedepth is an important parameter in structural graph theory (see \cite{nevsetvril2012sparsity}). 

We now explain how results in the literature immediately imply that a graph with sufficiently large pathwidth contains a bipartite subgraph with large pathwidth and a graph with sufficiently large treedepth contains a bipartite subgraph with large treedepth. To a certain degree, these parameters are simpler to deal with since trees have arbitrarily large pathwidth and treedepth. So it suffices to show that a graph with large pathwidth contains a subtree with large pathwidth and that a graph with large treedepth contains a subtree with large treedepth. For pathwidth, the excluded forest theorem demonstrates this.

\begin{thm}[\cite{robertson1983graph,bienstock1991quickly}]\label{pwExcludingForest}
	For every forest $F$, every graph of pathwidth at least $|V(F)|-1$ contains $F$ as a minor.
\end{thm}
The complete binary tree $T_k$ with height $k$ has $2^{k}-1$ vertices and pathwidth $\ceil{k/2}$ \cite{scheffler1989die}. Moreover, since $T_k$ has maximum degree $3$, every graph that contains $T_k$ as a minor also contains $T_k$ as a topological-minor. Hence \Cref{pwExcludingForest} implies the following.

\begin{cor}\label{pwbipartite}
	For every integer $k \geq 1$, every graph $G$ with $\pw(G)\geq 2^{(2k+1)}-2$ contains a bipartite subgraph $\hat{G}\sse G$ such that $\pw(\hat{G})\geq k$.
\end{cor}

For treedepth, \citet{nevsetvril2012sparsity} proved that every graph with treedepth at least $n$ contains an $n$-vertex path, $P_n$, as a subgraph. They also showed that the treedepth of $P_n$ is at least $\ceil{\log_2(n+1)}$. Hence their results imply the following.

\begin{cor}\label{tdbipartite}
	For every integer $n \geq 1$, every graph $G$ with $\td(G)\geq 2^n-1$ contains a bipartite subgraph $\hat{G}\sse G$ such that $\td(\hat{G})\geq n$.
\end{cor}

Again, it is open whether the bounds for pathwidth and treedepth in \Cref{pwbipartite,tdbipartite} can be improved to polynomial functions. 

We now explain how a polynomial bound for treewidth would imply polynomial bounds for both pathwidth and treedepth. We make use of the following two theorems from the literature. The first is by \citet{groenland2021approximating}.

\begin{thm}[\cite{groenland2021approximating}]\label{SmallTwLargePw}
	For all integers $k,h\geq 1$, every graph $G$ with $\tw(G)\geq k-1$ has $\pw(G)\leq kh+1$ or contains a subdivision of a complete binary tree of height $h$.
\end{thm}

 For treedepth, \citet{kawarabayashi2018polynomial} first demonstrated a polynomial excluded-minor approximation for graphs with large treedepth which was later improved by \citet{czerwinski2019improved}.
\begin{thm}[\cite{kawarabayashi2018polynomial,czerwinski2019improved}]\label{SmallTwLargeTd}
	There exists a positive constant $c$ such that for every integer $k \geq 1$ and graph $G$ with $\td(G)\geq ck^3$, either $\tw(G) \geq k$, or $G$ contains a subdivision of a complete binary tree of depth $k$ as a subgraph, or $G$ contains a path of length $2^k$.
\end{thm}

\begin{lem}
	Suppose there exists a polynomial function $f$ such that every graph $G$ with $\tw(G)\geq f(k)$ contains a bipartite subgraph $\hat{G}\sse G$ with $\tw(\hat{G})\geq k$. Then there exists polynomial functions $g$ and $h$ such that:
	\begin{compactenum}
		\item Every graph $G$ with $\pw(G)\geq g(k)$ contains a bipartite subgraph $\hat{G}\sse G$ such that $\pw(\hat{G})\geq k$; and
		\item Every graph $G$ with $\td(G)\geq h(k)$ contains a bipartite subgraph $\hat{G}\sse G$ such that $\td(\hat{G})\geq k$.		
	\end{compactenum}
\end{lem}

\begin{proof}
	Let $g(k):=2f(k)k+1$ and $h(k):=cf(k)^3$ where $c$ is from \Cref{SmallTwLargeTd}. Let $G$ be a graph. If $\tw(G)\geq f(k)$ then there exists a bipartite subgraph $\hat{G}\sse G$ such that $\tw(\hat{G})\geq k$. Since treedepth is bounded from below by pathwidth and pathwidth is bounded from below by treewidth, we are done. Therefore, we may assume that $\tw(G)\leq f(k)$. Now if $\pw(G)\geq g(k)$, then by \Cref{SmallTwLargePw}, $G$ contains a subdivision of the complete binary tree $T_{2k}$ as a subgraph. As $\pw(T_{2k})=h$, by setting $\hat{G}$ to be the subdivision of $T_{2k}$ in $G$, it follows that $\hat{G}$ is bipartite and $\pw(\hat{G})\geq k$. Now if $\td(G)\geq h(k)$, then $G$ either contains a subdivision of a complete binary tree of depth $2k$ or a path of length $2^k$. In either case, by setting $\hat{G}$ to be the subdivision of the complete binary tree of depth $2k$ or the path of length $2^k$, $\hat{G}$ is bipartite and we have $\td(\hat{G})\geq k$.
\end{proof}

\section{Conclusion}
We conclude with the following question. Say a graph parameter $\beta$ is \emph{unbounded on bipartite graphs} if for every integer $k \geq 0$, there exists a bipartite graph $G$ such that $\beta(G)\geq k$. For every such parameter $\beta$, does there exists a function $f$ such that every graph $G$ with $\beta(G)\geq f(k)$ contains a bipartite subgraph $\hat{G}\sse G$ with $\beta(\hat{G})\geq k$?
\makeatletter
\renewcommand\subsection{\@startsection{subsection}{3}{\z@}%
	{-3.25ex\@plus -1ex \@minus -.2ex}%
	{-1.5ex \@plus -.2ex}
	{\normalfont\normalsize\bfseries}}
\makeatother
\subsection*{Note:} After the first release of this paper, Gwena\"{e}l Joret and Piotr Micek informed us that our main results were already known. In particular, \citet{thomassen1988disjoint} proved a generalisation of \Cref{BipartiteSubgraphWall} and \citet{GGRS2009oddminor} proved \Cref{MainHadwigerBipartiteSubgraph}. Furthermore, Alex Scott answered our concluding question in the negative with the following construction. For an odd integer $n \geq 1$, let $C_n$ be a cycle on $n$ vertices. For a non-empty graph $G$, let $\textsf{a}(G)$ be the maximum odd integer $n\geq 1$ such that $C_n$ is a subgraph of $G$ and let $\textsf{b}(G)$ be the maximum integer $t \geq 0$ such that $G$ contains $K_{t,t}$ as a subgraph. Define $\textsf{c}(G):=\max\{\textsf{a}(G),\textsf{b}(G)\}$. Then $\textsf{c}(G)$ is unbounded on bipartite graphs. However, $\textsf{c}(C_n)=n$ for every odd integer $n \geq 1$ and $\textsf{c}(\hat{G})\leq 1$ whenever $\hat{G}$ is a bipartite subgraph of $C_n$, as required. Many thanks to Gwena\"{e}l Joret, Piotr Micek and Alex Scott for this feedback.

\fontsize{10}{11} 
\selectfont 
\let\oldthebibliography=\thebibliography
\let\endoldthebibliography=\endthebibliography
\renewenvironment{thebibliography}[1]{%
	\begin{oldthebibliography}{#1}%
		\setlength{\parskip}{0.3ex}%
		\setlength{\itemsep}{0.3ex}%
}{\end{oldthebibliography}}
	
\bibliographystyle{DavidNatbibStyle}
\bibliography{RobReferences}

\end{document}